\documentclass[12pt]{amsart}
\usepackage[T1]{fontenc}
\usepackage[latin9]{inputenc}
\usepackage{geometry}
\usepackage{mathrsfs}
\usepackage{amstext}
\usepackage{amsthm}
\usepackage{amssymb}

\usepackage{comment}

\makeatletter
\numberwithin{equation}{section}
\numberwithin{figure}{section}
\theoremstyle{plain}
\newtheorem{thm}{\protect\theoremname}[section]
\theoremstyle{definition}
\newtheorem{defn}[thm]{\protect\definitionname}
\theoremstyle{definition}
\newtheorem{question}{Question}[section]

\newtheorem{rem}[thm]{\protect\remarkname}
\theoremstyle{plain}

\theoremstyle{plain}
\newtheorem{lem}[thm]{\protect\lemmaname}
\theoremstyle{plain}
\newtheorem{prop}[thm]{\protect\propositionname}
\theoremstyle{definition}
\newtheorem{example}[thm]{\protect\examplename}

\usepackage[svgnames]{xcolor}
\usepackage[bookmarksnumbered=true]{hyperref} 
\hypersetup{
	colorlinks = true,
	linkcolor = Blue,
	anchorcolor = blue,
	citecolor = Green,
	filecolor = blue,
	urlcolor = FireBrick
}
\usepackage{bbm}

\MakeRobust{\ref}
\newcommand{\labeltext}[2]{
\@bsphack
\csname phantomsection\endcsname
\def\@currentlabel{#1}{\label{#2}}
\@esphack
}

\usepackage{scalerel}
\usepackage[usestackEOL]{stackengine}
\def\dashint{\,\ThisStyle{\ensurestackMath{%
  \stackinset{c}{.2\LMpt}{c}{.5\LMpt}{\SavedStyle-}{\SavedStyle\phantom{\int}}}%
  \setbox0=\hbox{$\SavedStyle\int\,$}\kern-\wd0}\int}

\usepackage{enumitem}

\DeclareRobustCommand{\SkipTocEntry}[5]{}

\newcommand{\mR}{\mathbb{R}}   
\newcommand{\mC}{\mathbb{C}}   
\newcommand{\abs}[1]{\lvert #1 \rvert}  
\newcommand{\norm}[1]{\lVert #1 \rVert}  

\newcommand{\ol}[1]{\overline{#1}}

\newcommand{\eps}{\epsilon}

\makeatother

\usepackage{babel}
\providecommand{\corollaryname}{Corollary}
\providecommand{\definitionname}{Definition}
\providecommand{\lemmaname}{Lemma}
\providecommand{\propositionname}{Proposition}
\providecommand{\remarkname}{Remark}
\providecommand{\theoremname}{Theorem}
\providecommand{\examplename}{Example}

\begin{document}


\title[On positivity sets for Helmholtz solutions]{On positivity sets for Helmholtz solutions}

\dedicatory{Dedicated to Carlos E.\ Kenig on the occasion of his 70$^{\text{th}}$ birthday}

\begin{sloppypar}

\begin{abstract}
We address the question of finding global solutions of the Helmholtz equation that are positive in a given set. This question arises in inverse scattering for penetrable obstacles. In particular, we show that there are solutions that are positive on the boundary of a bounded Lipschitz domain.
\end{abstract}

\author{Pu-Zhao Kow}
\address{Department of Mathematics and Statistics, P.O. Box 35 (MaD), FI-40014
University of Jyv\"{a}skyl\"{a}, Finland.}
\email{\href{mailto:pu-zhao.pz.kow@jyu.fi}{pu-zhao.pz.kow@jyu.fi}}
\author{Mikko Salo}
\address{Department of Mathematics and Statistics, P.O. Box 35 (MaD), FI-40014
University of Jyv\"{a}skyl\"{a}, Finland.}
\email{\href{mailto:mikko.j.salo@jyu.fi}{mikko.j.salo@jyu.fi}}
\author{Henrik Shahgholian}
\address{Department of Mathematics, KTH Royal Institute of Technology, SE-10044 Stockholm, Sweden. }
\email{\href{mailto:henriksh@kth.se}{henriksh@kth.se}}

\maketitle


\section{Introduction}




The objective in this short note is to consider the following problem.

\begin{question} \label{q_main}
Let $k > 0$ and let $E$ be a subset of $\mR^n$ {\rm (}$n \ge 2${\rm )}. Does there exist a solution of $(\Delta+k^2) u = 0$ in $\mR^n$ with $u|_E > 0$?
\end{question}

Note that any solution of the Helmholtz equation $(\Delta+k^2)u = 0$ is $C^{\infty}$, and thus the condition $u|_E > 0$ can be understood pointwise. There is a substantial literature on zero sets of solutions of elliptic equations and eigenfunctions, as discussed in the review \cite{LogunovMalinnikova2019}. In our setting, any real valued solution of $(\Delta+k^2)u = 0$ in $\mR^n$ must have a zero in any closed ball of radius $j_{\frac{n-2}{2},1} k^{-1}$ where $j_{\frac{n-2}{2},1}$ is the first zero of the Bessel function $J_{\frac{n-2}{2}}$ (see e.g.\  \cite[Lemma 3.1]{SS21NonscatteringFreeBoundary}). Question \ref{q_main} above is related to producing a global solution whose zero set avoids a given set $E$.

Our motivation comes from inverse scattering theory and the works \cite{CakoniVogelius, SS21NonscatteringFreeBoundary, KLSS22QuadratureDomain}. In these works, one considers a bounded open set $D \subset \mR^n$ (penetrable obstacle) together with a coefficient  $h \in L^{\infty}(\mR^n)$ with $\abs{h} \geq c > 0$ a.e.\ near $\partial D$ (contrast), and asks whether it is possible to find a solution $u_0 \not\equiv 0$ of $(\Delta+k^2)u_0=0$ in $\mR^n$ (incident wave) such that the obstacle $D$ with contrast $h$ does not produce any scattering response. The last condition can be precisely formulated as the existence of a function $u$ solving 
\begin{gather*}
    (\Delta + k^2 + h \chi_D) u = 0 \text{ in $\mR^n$}, \\
    u = u_0 \text{ outside some ball}.
\end{gather*}
If this happens for some contrast $h$, then the obstacle $D$ is called a \emph{non-scattering domain} and it will be invisible with respect to probing with the incident wave $u_0$.

It was proved in \cite[Theorem 2.1]{SS21NonscatteringFreeBoundary} that if $D$ has real-analytic boundary and if there is an incident wave $u_0$ with $u_0|_{\partial D} > 0$, then $D$ is a non-scattering domain. Similarly, the work \cite{KLSS22QuadratureDomain} introduced the notion of quadrature domains for the Helmholtz operator $\Delta+k^2$ and proved that if $D$ is such a domain, and if there is an incident wave $u_0$ with $u_0|_{\partial D} > 0$, then $D$ is a non-scattering domain. On the other hand, the works \cite{CakoniVogelius, SS21NonscatteringFreeBoundary} show that under a nonvanishing condition for $u_0$ on $\partial D$, the boundary of a non-scattering domain can be interpreted as a free boundary in an obstacle-type problem and hence such a domain must be either regular or have thin complement near any boundary point.

It was also proved in \cite{SS21NonscatteringFreeBoundary} that one may be able to find incident waves that are positive on the boundary of a bounded $C^1$ domain (Lipschitz if $n=2,3$). Our first main result extends this to Lipschitz domains in any dimension.

\begin{thm}
\label{thm:main-incident}Let $D\subset\mathbb{R}^{n}$ {\rm (}$n \ge 2${\rm )} be a bounded Lipschitz domain such that $\mathbb{R}^{n}\setminus\overline{D}$ is connected. Suppose that $k^{2}>0$ is not a Dirichlet eigenvalue of $-\Delta$ in $D$. Then there exists a Herglotz wave function $u_{0}$ {\rm (}see Definition~{\rm \ref{def:Herglotz-wave}}{\rm )} satisfying  
\[
(\Delta+k^{2})u_{0}=0\text{ in }\mathbb{R}^{n}\text{ and }u_{0}|_{\partial D}>0.
\]
\end{thm}

The proof of Theorem \ref{thm:main-incident} is done in two steps. One first constructs a solution $v$ of $(\Delta+k^2) v = 0$ in $D$ with $v|_{\partial D} > 0$ by solving a Dirichlet problem. Then one approximates $v$ in $D$ by a suitable  Herglotz wave $u_0$ in $\mR^n$ via a Runge approximation argument. This approximation needs to be done in a suitable norm to obtain the pointwise condition $u_0|_{\partial D} > 0$, but since $D$ only has Lipschitz boundary the solution $v$ is not very regular and this limits the choice of possible norms. We will work with fractional Sobolev spaces $H^{s,p}$ and invoke the theory of boundary value problems in Lipschitz domains.\footnote{This is one of the areas where Carlos Kenig has made pioneering contributions.} 

We remark that the assumption in Theorem \ref{thm:main-incident} that $k^2$ is not an eigenvalue is necessary, at least when $D$ is a ball (see Example \ref{ex_dirichlet_eigenvalue}). For the first eigenvalue this was pointed out in \cite[Remark 3.2]{SS21NonscatteringFreeBoundary}.

Another instance of subsets $E \subset \mR^n$ where one can arrange $u_0|_E > 0$ is given in the following result.

\begin{thm}
\label{thm:main-incident-two} Let $k > 0$, and let $D\subset\mathbb{R}^{n}$ {\rm (}$n \ge 2${\rm )} be a bounded Lipschitz domain such that $\mathbb{R}^{n}\setminus\overline{D}$ is connected and $\abs{D} \leq \abs{B_r}$ where $r = j_{\frac{n-2}{2},1} k^{-1}$. If $E \subset D$ is compact, then there exists a Herglotz wave function $u_{0}$ {\rm (}see Definition~{\rm \ref{def:Herglotz-wave}}{\rm )} satisfying  
\begin{equation}
(\Delta+k^{2})u_{0}=0\text{ in }\mathbb{R}^{n}\text{ and }u_{0}|_{E}>0. \label{eq:positive-solution-on-E}
\end{equation}
\end{thm}

The proof is similar to that of Theorem \ref{thm:main-incident}, except that in the first step we use the Faber-Krahn inequality to produce a solution $v$ that is positive near $E$.

\begin{rem}
If $E$ is sufficiently nice and low dimensional, it may be possible to use Theorem \ref{thm:main-incident-two} to find solutions that are positive on $E$. For example, let $E$ be a smooth compact manifold with $\dim\,(E) = m \le n-2$ embedded in $\mathbb{R}^{n}$, which is homeomorphic to a compact submanifold $E_{1}$ of $\mathbb{R}^{n-1} \cong \mathbb{R}^{n-1} \times \{0\} \subset \mathbb{R}^{n}$. This holds e.g.\ when $m < n/2$ by the Whitney embedding theorem, or when $E$ is homeomorphic to $S^m$. Since $\mathbb{R}^{n}\setminus E_{1}$ is connected, by \cite[Corollary~7.9]{MadsenTornehave} one sees that $\mathbb{R}^{n}\setminus E$ is (pathwise) connected. One can construct a tubular neighborhood $D = \{ x \in \mR^n \,:\, d(x,E) < \eps \}$ of $E$ having smooth boundary $\partial D$ and arbitrarily small measure \cite[Theorem~9.23 and Remark~9.24]{MadsenTornehave} (see also \cite[Theorem~6.24]{John13IntroductionSmoothManifold}). Since $\mR^n \setminus E$ is connected, one can connect any two points in $\mR^n \setminus D$ by a curve $\gamma$ in $\mR^n \setminus E$. By considering the curve $F(\gamma)$ where $F$ is a continuous map on $\mR^n$ that fixes $\mR^n \setminus D$ and collapses $D \setminus E$ to $\partial D$, we see that $\mR^n \setminus D$ is connected. Since $D$ has smooth boundary, also $\mR^n \setminus \overline{D}$ is connected. (See \cite[pages~61--62]{CF77KnotTheory} for a related discussion.) Thus we may apply Theorem~{\rm \ref{thm:main-incident-two}} to find a Herglotz wave function $u_{0}$ satisfying \eqref{eq:positive-solution-on-E}. Note that the connectedness of $\mR^n \setminus E$ can fail when $E$ has dimension $n-1$. 
\end{rem}


\section{\label{sec:incident-field}Solutions satisfying the positivity condition} 

In this section we will prove Theorems~\ref{thm:main-incident} and \ref{thm:main-incident-two}. We begin with some preparations.

\addtocontents{toc}{\SkipTocEntry}
\subsection{Fractional Sobolev spaces}

For each $s\in\mathbb{R}$ and $1<p<\infty$, the fractional Sobolev space $H^{s,p}(\mathbb{R}^{n})$
is the Banach space equipped with the norm 
\[
\|u\|_{H^{s,p}(\mathbb{R}^{n})}:=\|\langle D\rangle^{s}u\|_{L^{p}(\mathbb{R}^{n})},
\]
where $\langle D\rangle^{s}$ is the the \emph{Bessel potential} of order $s$, i.e.\ the Fourier multiplier corresponding
to $\langle\xi\rangle^{s}=(1+|\xi|^{2})^{\frac{s}{2}}$. 
In particular when $s=k\ge1$ is an integer, we also have $H^{k,p}(\mathbb{R}^{n})=W^{k,p}(\mathbb{R}^{n})$,
where 
\[
W^{k,p}(\mathbb{R}^{n})=\begin{Bmatrix}\begin{array}{l|l}
u\in L^{p}(\mathbb{R}^{n}) & D^{\alpha}u\in L^{p}(\mathbb{R}^{n})\text{ for all multi-indices }\alpha\text{ with }|\alpha| \leq k\end{array}\end{Bmatrix}.
\]
From \cite[Corollary~6.2.8]{BL76InterpolationSpaces}, we have the duality statement 
\begin{equation}
(H^{s,p}(\mathbb{R}^{n}))^{*}=H^{-s,p'}(\mathbb{R}^{n})\quad\text{for all }s\in\mathbb{R}\text{ and }1<p<\infty,\label{eq:duality}
\end{equation}
where $(p')^{-1}+p^{-1}=1$. We also recall the Sobolev embedding (\cite[Theorem~6.5.1]{BL76InterpolationSpaces}): 
\begin{equation} \label{eq_fse}
    H^{s,p}(\mathbb{R}^{n})\subset H^{s_{1},p_{1}}(\mathbb{R}^{n})
\end{equation}
whenever $1<p\le p_{1}<\infty$, $-\infty<s_{1}\le s<\infty$, and  $s-\frac{n}{p}=s_{1}-\frac{n}{p_{1}}$. 

Let $D$ be an open set in $\mathbb{R}^{n}$. We define 
\[
H^{s,p}(D):=\begin{Bmatrix}\begin{array}{l|l}
u|_{D} & u\in H^{s,p}(\mathbb{R}^{n})\end{array}\end{Bmatrix}\quad  \text{for all  }s\in\mathbb{R}\text{ and }1<p<\infty.
\] 
This is a Banach space equipped with the quotient norm 
\[
\|v\|_{H^{s,p}(D)}:=\inf\begin{Bmatrix}\begin{array}{l|l}
\|u\|_{H^{s,p}(\mathbb{R}^{n})} & u|_{D}=v\end{array}\end{Bmatrix}.
\]
When $D$ is a bounded Lipschitz domain, from \cite[Theorem~2.3]{JK95Dirichlet}
we know that there exists a bounded linear extension operator 
\begin{equation}
E:H^{s,p}(D)\rightarrow H^{s,p}(\mathbb{R}^{n})\quad\text{with }Eu=u\text{ in }D\text{ for all }u\in H^{s,p}(D).
\label{eq:Hsp-extension}
\end{equation}
If $F \subset \mR^n$ is closed, we define 
\[
H_{F}^{s,p}(\mR^n):=\begin{Bmatrix}\begin{array}{l|l}
u\in H^{s,p}(\mathbb{R}^{n}) & \mathrm{supp}(u) \subset F \end{array}\end{Bmatrix}.
\]
If $D$ is a bounded Lipschitz domain, the following result can be found in \cite[Remark~2.7]{JK95Dirichlet}: 
\begin{equation} \label{eq_density_sobolev}
\text{$C_{c}^{\infty}(D)$ is dense in $H_{\ol{D}}^{s,p}(\mR^n)$
for each $s \in \mR$ and $1<p<\infty$.}
\end{equation}



\addtocontents{toc}{\SkipTocEntry}
\subsection{Runge-Herglotz approximation}

The next objective is to prove a result stating that solutions in $H^{s,p}(D)$ can be approximated in $D$ by Herglotz waves. We first give a definition.

\begin{defn}\label{def:Herglotz-wave}
Let $k>0$ and consider the operator $P_{k}:C^{\infty}(\mathcal{S}^{n-1})\rightarrow C^{\infty}(\mathbb{R}^{n})$ defined by  
\begin{equation*}
(P_{k}f)(x):=\int_{\mathcal{S}^{n-1}}e^{ikx\cdot\hat{z}}f(\hat{z})\,d\hat{z}, \qquad x\in\mathbb{R}^{n}.
\end{equation*}
The functions $u=P_{k}f$ with $f \in C^{\infty}(\mathcal{S}^{n-1})$ are called \emph{Herglotz waves}, and they are particular solutions of $(\Delta+k^{2})u=0$ in $\mathbb{R}^{n}$.
\end{defn}


\begin{prop}
\label{prop:Herglotz} Let $k>0$, $0<s\le1$, $1<p<\infty$, and
let $D\subset\mathbb{R}^{n}$ {\rm (}$n \ge 2${\rm )} be a bounded Lipschitz domain such that
$\mathbb{R}^{n}\setminus\overline{D}$ is connected. Given any $v\in H^{s,p}(D)$
with $(\Delta+k^{2})v=0$ in $D$, there exist Herglotz
waves $u_{j}\in C^{\infty}(\mathbb{R}^{n})$ such that 
\[
\|u_{j}-v\|_{H^{s,p}(D)}\rightarrow0\quad\text{as }j\rightarrow\infty.
\]
If $v$ is real-valued, then so are $u_{j}$. 
\end{prop}

The proof of Proposition~\ref{prop:Herglotz} is very similar to \cite[Proposition~3.4]{SS21NonscatteringFreeBoundary} that considered approximation in $W^{1,p}(D)$. Here we need to work with fractional Sobolev spaces instead.


\begin{proof}
In view of the  Hahn-Banach
theorem, it is enough to prove that any bounded linear functional
$\ell: H^{s,p}(D) \to \mC$ that vanishes on $\begin{Bmatrix}\begin{array}{l|l}
P_{k}f|_{D} & f\in C^{\infty}(\mathcal{S}^{n-1})\end{array}\end{Bmatrix}$ must also vanish on $\begin{Bmatrix}\begin{array}{l|l}
v\in H^{s,p}(D) & -(\Delta+k^{2})v=0\text{ in }D\end{array}\end{Bmatrix}$. Let $\ell$ be such a linear functional, and define a bounded linear functional $\ell_{1}:H^{s,p}(\mathbb{R}^{n})\rightarrow\mathbb{C}$
by $\ell_{1}(u):=\ell(u|_{D})$. By duality \eqref{eq:duality}, there
exists a unique $\mu\in H^{-s,p'}(\mathbb{R}^{n})$ such that 
\[
\ell_{1}(u)=(u,\mu)\quad\text{for all }u\in H^{s,p}(\mathbb{R}^{n}),
\]
where $(\cdot,\cdot)$ is the sesquilinear distributional pairing
in $\mathbb{R}^{n}$. It is easy to see that $\mu=0$ in $\mathbb{R}^{n}\setminus\overline{D}$,
and the condition $\ell(P_{k}f|_{D})=0$ for all $f\in C^{\infty}(\mathcal{S}^{n-1})$
implies that 
\begin{equation}
(P_{k}f,\mu)=0\quad\text{for all }f\in C^{\infty}(\mathcal{S}^{n-1}).\label{eq:ell1-cond1}
\end{equation}

We now define the distribution $w:=\Phi_{k}*\mu$, where 
\begin{equation}
\Phi_{k}(x)=\frac{ik^{\frac{n-2}{2}}}{4(2\pi)^{\frac{n-2}{2}}}|x|^{-\frac{n-2}{2}}H_{\frac{n-2}{2}}^{(1)}(k|x|)\label{eq:outgoing-fundamental-solution}
\end{equation}
is the outgoing fundamental solution of the Helmholtz operator $-(\Delta+k^{2})$ and $H_{\alpha}^{(1)}$ is the Hankel function (see \cite[\S 1.2.3]{Yaf10ScatteringAnalyticTheory}). Then $w$ is a distributional solution of 
\begin{equation}
-(\Delta+k^{2})w=\mu\quad\text{in }\mathbb{R}^{n}.\label{eq:def-w-aux}
\end{equation}
Elliptic regularity yields $w \in H^{2-s,p'}_{\mathrm{loc}}(\mR^n)$, and since ${\rm supp}(\mu)\subset\overline{D}$ we also have that $w$ is $C^{\infty}$ in $\mathbb{R}^{n}\setminus\overline{D}$.

Given any $f\in C^{\infty}(\mathcal{S}^{n-1})$, we write $u=P_{k}f\in C^{\infty}(\mathbb{R}^{n})$.
Using \eqref{eq:ell1-cond1} and the fact that $\mu$ has compact support,
we have  
\begin{equation}
0=(u,\mu)=\lim_{r\rightarrow\infty}(u,\mu)_{B_{r}},\label{eq:ell1-cond2}
\end{equation}
where $(\cdot,\cdot)_{B_{r}}$ is the sesquilinear distributional
pairing in the ball $B_{r}$. We now consider a cut-off function
$\chi\in C_{c}^{\infty}(\mathbb{R}^{n})$ satisfying $0\le\chi\le1$
and $\chi=1$ near $\overline{D}$. Using \eqref{eq:def-w-aux},
we can write \eqref{eq:ell1-cond2} as 
\begin{align}
0 & =\lim_{r\rightarrow\infty}\bigg[(\chi u,(\Delta+k^{2})w)_{B_{r}}+((1-\chi)u,(\Delta+k^{2})w)_{B_{r}}\bigg]\nonumber \\
 & =\lim_{r\rightarrow\infty}\bigg[((\Delta+k^{2})(\chi u),w)_{B_{r}}+((\Delta+k^{2})((1-\chi)u),w)_{B_{r}}\nonumber \\
 & \qquad\qquad+\int_{\partial B_{r}}(u\overline{\partial_{|x|}w}-(\partial_{|x|}u)\overline{w})\,dS\bigg]\nonumber \\
 & =\lim_{r\rightarrow\infty}\int_{\partial B_{r}}(u\overline{\partial_{|x|}w}-(\partial_{|x|}u)\overline{w})\,dS,\label{eq:ell1-cond3}
\end{align}
where $\partial_{|x|}=\hat{x}\cdot\nabla$ denotes the radial derivative. Here we also used the fact that $(\Delta+k^{2})u=0$ in $\mathbb{R}^{n}$. 

\begin{subequations}
Using \cite[Lemma~1.2 and equation~(1.18)]{Mel95GeometricScattering}, 
we know that the Herglotz function $u=P_{k}f$ has the following asymptotics as $|x| \rightarrow \infty$: 
\begin{align}
u(x) & =c_{n,k}'|x|^{-\frac{n-1}{2}}\bigg(e^{ik|x|}f(\hat{x})+i^{n-1}e^{-ik|x|}f(-\hat{x})\bigg)+O(|x|^{-\frac{n+1}{2}}),\label{eq:asymptotic1a}\\
\partial_{|x|}u(x) & =c_{n,k}'|x|^{-\frac{n-1}{2}}ik\bigg(e^{ik|x|}f(\hat{x})-i^{n-1}e^{-ik|x|}f(-\hat{x})\bigg)+O(|x|^{-\frac{n+1}{2}}),
\end{align}
where $c_{n,k}'=k^{\frac{n-1}{2}}e^{\frac{\pi(n-1)i}{4}}(2\pi)^{-\frac{n-1}{2}}$.
On the other hand, from \cite[equation~(2.27)]{Yaf10ScatteringAnalyticTheory},
we know that $w$ has the asymptotics 
\begin{align}
w(x) & =c_{n,k}''|x|^{-\frac{n-1}{2}}e^{ik|x|}\hat{\mu}(k\hat{x})+O(|x|^{-\frac{n+1}{2}})\quad\text{as }|x|\rightarrow\infty,\label{eq:asymptotic1c}\\
\partial_{|x|}w(x) & =c_{n,k}''|x|^{-\frac{n-1}{2}}ike^{ik|x|}\hat{\mu}(k\hat{x})+O(|x|^{-\frac{n+1}{2}})\quad\text{as }|x|\rightarrow\infty,\label{eq:asymptotic1d}
\end{align}
where $c_{n,k}''=2^{-1}e^{-\frac{\pi(n-3)i}{4}}(2\pi)^{-\frac{n-1}{2}}k^{\frac{n-3}{2}}$
and $\hat{\mu}\in C^{\infty}(\mathbb{R}^{n})$ is the Fourier transform
of the compactly supported distribution $\mu$.
\end{subequations}

\begin{subequations}
Combining \eqref{eq:ell1-cond3} with \eqref{eq:asymptotic1a}--\eqref{eq:asymptotic1d},
we obtain 
\[
\int_{\mathcal{S}^{n-1}}f(\hat{x})\overline{\hat{\mu}(k\hat{x})}\,d\hat{x}=0.
\]
By the fact that $f\in C^{\infty}(\mathcal{S}^{n-1})$ was arbitrary, we conclude
$\hat{\mu}(k\hat{x})=0$ for all $\hat{x}\in\mathcal{S}^{n-1}$. Consequently,
\eqref{eq:asymptotic1c} becomes 
\[
w(x)=O(|x|^{-\frac{n+1}{2}})\quad\text{as }|x|\rightarrow\infty.
\]
In other words, the far-field pattern of $w$ is vanishing. By the Rellich uniqueness theorem \cite{CK19scattering,Hormander_rellich}, the unique continuation principle and the connectedness of $\mR^n \setminus \ol{D}$, we conclude that 
\begin{equation}
w=0\quad\text{in }\mathbb{R}^{n}\setminus\overline{D}.\label{eq:support-w-aux}
\end{equation}
Since $w \in H^{2-s,p'}_{\mathrm{loc}}(\mR^n)$, we also conclude that $w\in H^{2-s,p'}_{\ol{D}}(\mathbb{R}^{n})$. 
\end{subequations}

Now let $v\in H^{s,p}(D)$ be any solution of $(\Delta+k^{2})v=0$
in $D$, and let $\tilde{v}\in H^{s,p}(\mathbb{R}^{n})$ be such that
$\tilde{v}|_{D}=v$. We see that 
\[
\ell(v)=\ell_{1}(\tilde{v}|_{D})=(\tilde{v},\mu)=(\tilde{v},(\Delta+k^{2})w).
\]
From \eqref{eq_density_sobolev}, we know that there are $w_{j}\in C_{c}^{\infty}(D)$
with $w_{j}\rightarrow w$ in $H^{2-s,p'}(\mathbb{R}^{n})$. Since
$(\Delta+k^{2})\tilde{v}=0$ in $D$, we finally conclude that 
\[
\ell(v)=\lim_{j\rightarrow\infty} (\tilde{v},(\Delta+k^{2})w_j) = \lim_{j\rightarrow\infty}((\Delta+k^{2})\tilde{v},w_{j})=0,
\]
which is our desired result. 
\end{proof}

\addtocontents{toc}{\SkipTocEntry}
\subsection{Proof of the main result}

Theorem~\ref{thm:main-incident} 
is an immediate consequence of the following result: 
\begin{thm}
\label{thm:main1-strong}Let $D$ be a bounded Lipschitz domain in
$\mathbb{R}^{n}$ {\rm (}$n \ge 2${\rm )} such that $\mathbb{R}^{n}\setminus\overline{D}$
is connected. Suppose that $k^2>0$ is not a Dirichlet eigenvalue of
$-\Delta$ in $D$. Given any constant $c_{0}\in\mathbb{R}$, there
exist Herglotz wave functions $u_{j}\in C^{\infty}(\mathbb{R}^{n})$ solving $(\Delta + k^2) u_j = 0$ in $\mathbb{R}^n$ such
that 
\[
\lim_{j\rightarrow\infty}\|u_{j}-c_{0}\|_{L^{\infty}(\partial D)}=0.
\]
\end{thm}

Before we prove Theorem~\ref{thm:main1-strong} we need the following
result, which is a special case of \cite[Theorems~1.1~\&~1.3]{JK95Dirichlet}.

\begin{prop}
\label{prop:elliptic-regularity} Let $D$ be a bounded Lipschitz domain
in $\mathbb{R}^{n}$ {\rm (}$n \ge 2${\rm )}. If $2 \le p < \infty$ and $f\in H^{s-2,p}(D)$
where 
\[
\frac{1}{p}<s<\frac{3}{p},
\]
then there exists a unique $u\in H^{s,p}(D)$ satisfying $-\Delta u=f$
in $D$ and $u=0$ on $\partial D$. 
\end{prop}
\begin{proof}
We first consider the case when $n \ge 3$. Let $p_0$ be as in \cite[Theorem~1.1]{JK95Dirichlet} (with $\Omega = D$). If $p_0' \le p < \infty$, the result follows from \cite[Theorem~1.1(c)]{JK95Dirichlet}. On the other hand, if $2 \le p < p_0'$, the result follows from \cite[Theorem~1.1(a)]{JK95Dirichlet} since $s < \frac{3}{p} \le 1 + \frac{1}{p}$. The case when $n=2$ can be proved using identical reasoning using \cite[Theorem~1.3]{JK95Dirichlet} and the observation $\frac{3}{p} \le \frac{2}{p}+\frac{1}{2}$. 
\end{proof}


\begin{proof}
[Proof of Theorem~{\rm \ref{thm:main1-strong}}] Since $k^{2}$ is
not a Dirichlet eigenvalue in $D$, there exists a unique solution
$v\in H^{1,2}(D)$ such that 
\[
(\Delta+k^{2})v=0\text{ in }D\quad\text{and}\quad v=c_{0}\text{ on }\partial D.
\]
If $v\in H^{s,p}(D)$ for some $0<s\le1$ and $p>n/s$, using Proposition~\ref{prop:Herglotz},
we know that there exist Herglotz waves $u_{j}\in C^{\infty}(\mathbb{R}^{n})$ such
that 
\[
\|u_{j}-c_{0}\|_{L^{\infty}(\partial D)}=\|u_{j}-v\|_{L^{\infty}(\partial D)}\le\|u_{j}-v\|_{C(\overline{D})}\le C\|u_{j}-v\|_{H^{s,p}(D)}\rightarrow0,
\]
where we used the Sobolev embedding. 

It remains to show that $v\in H^{s,p}(D)$ for some $s, p$ with $s > n/p$, and this follows from a standard bootstrap argument based on Proposition \ref{prop:elliptic-regularity}. We claim that 
\begin{equation} \label{v_space}
v \in H^{\frac{2}{p_j},p_j}(D) \text{ for } 0 \leq j < \frac{n-2}{4},
\end{equation}
where 
\[
\frac{1}{p_j} = \frac{1}{2} - j \frac{2}{n-2}.
\]
The case $j=0$ follows since $v \in H^{1,2}(D)$. We argue by induction and assume that this holds for $j$. Define $w:=v-c_{0}$ and note that $w$ solves 
\[
-\Delta w=k^{2}v \in H^{\frac{2}{p_j}, p_j}(D), \qquad w|_{\partial D} = 0.
\]
We next use the Sobolev embedding $H^{\frac{2}{p_j}, p_j}(D) \subset H^{\frac{2}{q}-2, q}(D)$ where $\frac{2}{p_j} > \frac{2}{q}-2$ and 
\[
\frac{2}{p_j} - \frac{n}{p_j} = \frac{2}{q}-2 - \frac{n}{q}.
\]
It follows that $q = p_{j+1}$ and then indeed $\frac{2}{p_j} > \frac{2}{q}-2$. In particular $-\Delta w \in H^{\frac{2}{p_{j+1}}-2, p_{j+1}}(D)$ with $w|_{\partial D} = 0$, and we may use Proposition \ref{prop:elliptic-regularity} to conclude that $w \in H^{\frac{2}{p_{j+1}}, p_{j+1}}(D)$. This completes the induction step and proves \eqref{v_space}.

We have proved that $v \in H^{\frac{2}{p_j},p_j}(D)$ where $j$ is the largest integer $< \frac{n-2}{4}$. Using the above notation, we have $\Delta w \in H^{\frac{2}{p_j},p_j}(D)$ and $w|_{\partial D} = 0$. By Sobolev embedding we have $\Delta w \in H^{s-2,p}(D)$ whenever $p \geq p_j$ and \[
\frac{2}{p_j}-\frac{n}{p_j} = s - 2 - \frac{n}{p}.
\]
The last condition implies that 
\[
s-\frac{n}{p} = 2 + \frac{2-n}{p_j} = 2 + \frac{2-n}{2} + 2j \geq 0
\]
since $j \geq \frac{n-2}{4}-1$.
If $j > \frac{n-2}{4}-1$, using Proposition \ref{prop:elliptic-regularity} once again we obtain that $w$ and hence $v$ is in $H^{s,p}$ for some $s > n/p$. On the other hand, if $j = \frac{n-2}{4}-1$ we iterate the argument once more to get $v \in H^{s,p}$ for some $s > n/p$. This concludes the proof.
\end{proof}

The next simple example shows that the condition that $k^2$ is not an eigenvalue is necessary at least for balls.

\begin{example} \label{ex_dirichlet_eigenvalue}
Let $v(x) := |x|^{\frac{2-n}{2}} J_{\frac{n-2}{2}}(|x|)$. We see that
$v\in C^{\infty}(\mathbb{R}^{n})$ and $(\Delta+1)v=0$ in $\mathbb{R}^{n}$.
Suppose that $u_{1}$ is a real-valued function satisfying $(\Delta+1)u_{1}=0$
in $\mathbb{R}^{n}$. Since 
\[
v(x)=0\quad\text{when }|x|=j_{\frac{n-2}{2},m} \text{ for any $m \geq 1$},
\]
where $j_{\frac{n-2}{2},m}$ denotes the $m^{\rm th}$ positive zero of $J_{\frac{n-2}{2}}$, we have 
\[
\int_{|x|=j_{\frac{n-2}{2},m}}u_{1}\frac{\partial v}{\partial r}\,dS=\int_{|x|<j_{\frac{n-2}{2},m}}(u_{1}\Delta v-v\Delta u_{1})\,dx=0.
\]
Since  
\[
(-1)^{m}\frac{\partial v}{\partial r}(x)>0\quad\text{when }|x|=j_{\frac{n-2}{2},m},
\]
it follows that $u_{1}$ must change sign on $|x|=j_{\frac{n-2}{2},m}$.

Similarly, if $R > 0$ and if $u_0$ solves $(\Delta+k_{m}^{2})u_{0}=0$ in $\mathbb{R}^{n}$ where $k_m = R^{-1}j_{\frac{n-2}{2},m}$, define $u_1$ via the rescaling  
\[
u_{0}(x)=u_{1}(R^{-1}j_{\frac{n-2}{2},m}x)\quad\text{for }x\in\mathbb{R}^{n}.
\]
We see that $(\Delta+1)u_1 = 0$ in $\mR^n$. The above discussion
shows that $u_{0}$ must change sign on $\partial B_{R}$. 
\end{example}

The following strong maximum principle can be found in \cite[Appendix~A]{KLSS22QuadratureDomain}. However, for readers' convenience, here we exhibit the statement as well as its proof. 

\begin{lem}[Strong maximum principle] \label{lem:strong-maximum-principle}
Let $D$ be a bounded Lipschitz domain in $\mathbb{R}^{n}$ {\rm (}$n \ge 2${\rm )}, and let $k^{2} < \lambda_{1}(D)$, where $\lambda_{1}(D)>0$ denotes the smallest $H_{0}^{1}(D)$-eigenvalue of $-\Delta$. If the solution $u \in H^{1}(D)$ satisfies 
\[
(\Delta + k^{2})u = 0\text{ in $D$},\quad u\ge 0 \text{ on $\partial D$,}
\]
then for each open component $G$ of $D$ we have either $u \equiv 0$ in $G$ or $u > 0$ in $G$ {\rm (}note that $u \in C^{\infty}(G)$ by elliptic regularity{\rm )}.  
\end{lem}

\begin{proof}
It is easy to see that for each component $G$ of $D$ we have $k^{2} < \lambda_{1}(G)$ and 
\[
(\Delta + k^{2})u = 0\text{ in $G$},\quad u\ge 0 \text{ on $\partial G$.}
\]
Testing the equation above by $u_{-} \in H_{0}^{1}(G)$ and using Poincar\'{e} inequality, we have 
\[
\int_{G} |u_{-}|^{2} \,dx \le \frac{1}{\lambda_{1}(G)} \int_{G} |\nabla u_{-}|^{2} \,dx = \frac{k^{2}}{\lambda_{1}(G)} \int_{G} |u_{-}|^{2} \,dx.
\]
Since $\frac{k^{2}}{\lambda_{1}(G)}<1$, then $u_{-} \equiv 0$ in $G$, that is, 
\begin{equation}
u \ge 0 \text{ in $G$.} \label{eq:weak-maximum-principle}
\end{equation}
Let $x_{0} \in G$ such that $u(x_{0})=0$. The mean value theorem for Helmholtz equation (see e.g.\ \cite[Appendix~A]{KLSS22QuadratureDomain}) gives that 
\begin{equation}
\int_{B_{\epsilon}(x_{0})} u(x) \,dx = 0 \label{eq:MVT-zero-consequence}
\end{equation}
for all sufficiently small $\epsilon>0$ so that $\overline{B_{\epsilon}(x_{0})}\subset G$. Since $u$ is continuous in $G$, combining \eqref{eq:weak-maximum-principle} and \eqref{eq:MVT-zero-consequence} we know that $u=0$ in $B_{\epsilon}(x_{0})$, and this shows that $\begin{Bmatrix}\begin{array}{l|l} x \in G & u(x)=0 \end{array}\end{Bmatrix}$ is both open and closed in $G$. Since $G$ is connected, then we have either 
\[
\begin{Bmatrix}\begin{array}{l|l} x \in G & u(x)=0 \end{array}\end{Bmatrix} = G \quad \text{or} \quad \begin{Bmatrix}\begin{array}{l|l} x \in G & u(x)=0 \end{array}\end{Bmatrix} = \emptyset,
\]
which concludes our lemma. 
\end{proof}

Finally, we give the proof of Theorem \ref{thm:main-incident-two}.

\begin{proof}
[Proof of Theorem~{\rm \ref{thm:main-incident-two}}] 
Since $\abs{D} \leq \abs{B_r}$ where $r = j_{\frac{n-2}{2},1} k^{-1}$, the Faber-Krahn inequality (see
e.g. \cite[Theorem~III.3.1]{Cha01IsoperimetricInequality}) implies that each connected component $G$ of $D$ satisfies 
\[
\lambda_1(G) \geq \lambda_1(B_r) = k^2.
\]

\medskip

\noindent \textbf{Case 1.} If $\lambda_1(G) = k^2$, we choose $v$ to be an eigenfunction corresponding to the first eigenvalue with $v > 0$ in $G$, i.e.\ $v$ solves $(\Delta+k^2) v = 0$ in $G$ with $v \in H^1_0(G)$, see e.g.\ \cite[Theorem~2{\rm (ii)} in Section~6.5.1]{Eva10PDE}. 

\medskip

\noindent \textbf{Case 2.} If $\lambda_{1}(G) > k^{2}$, then there exists a unique solution $v \in H^{1}(G)$ such that 
\[
(\Delta + k^{2})v=0 \text{ in $G$},\quad v=1 \text{ on $\partial G$.}
\]
Using the strong maximum principle in Lemma~{\rm \ref{lem:strong-maximum-principle}}, we know that $v>0$ in $G$. 

\medskip

Next we choose a bounded Lipschitz domain $D_1$ that satisfies $E \subset D_1$, $\ol{D}_1 \subset D$, and $\mR^n \setminus \ol{D}_1$ is connected. The function $v|_{D_1}$ is in $H^{1,p}(D_1)$ for any $p > n$ and satisfies $v|_{\ol{D}_1} > 0$. The approximation result in Proposition \ref{prop:Herglotz} yields a sequence of Herglotz waves $u_j$ satisfying 
\[
\norm{u_j|_{D_1} - v}_{H^{1,p}(D_1)} \to 0 \text{ as $j \to \infty$.}
\]
If $j$ is sufficiently large, the Sobolev embedding ensures that $u_j|_{E} > 0$.
\end{proof}

\section*{Acknowledgments}

\noindent 
This project was finalized while the authors stayed at Institute Mittag Leffler (Sweden), during the program Geometric aspects of nonlinear PDE.
Kow and Salo were partly supported by the Academy of Finland (Centre of Excellence in Inverse Modelling and Imaging, 312121) and by the European Research Council under Horizon 2020 (ERC CoG 770924). Shahgholian was supported by Swedish Research Council. 

\section*{Declarations}

\noindent {\bf  Data availability statement:} All data needed are contained in the manuscript.

\medskip
\noindent {\bf  Funding and/or Conflicts of interests/Competing interests:} The authors declare that there are no financial, competing or conflict of interests.

\end{sloppypar}

\bibliographystyle{custom}
\bibliography{ref}
\end{document}